\documentclass[11pt]{amsart}
\usepackage{a4wide}
\usepackage[all,cmtip]{xy}
\usepackage{mathabx}
\usepackage{leftidx}
\usepackage{breqn}
\usepackage{tikz-cd}

\usepackage{amscd,amsmath,amssymb,fancyhdr,color}
\usepackage{tikz-cd}
\usepackage{scalerel, mathtools}
\usepackage[backref=page]{hyperref}
\renewcommand*{\backref}[1]{}
\renewcommand*{\backrefalt}[4]{%
    \ifcase #1 (Not cited.)%
    \or        (Cited on page~#2.)%
    \else      (Cited on pages~#2.)%
    \fi}

\newcommand{\CC}{\mathbb{C}}

\newcommand{\RR}{\mathbb{R}}

\newcommand{\NN}{\mathbb{N}}

\numberwithin{equation}{section}

\def\eqref#1{(\ref{#1})}

\newcommand{\N}{{\mathbb N}}
\newcommand{\Z}{{\mathbb Z}}
\newcommand{\C}{{\mathbb C}}
\newcommand{\R}{{\mathbb R}}
\newcommand{\Q}{{\mathbb Q}}
\renewcommand{\H}{{\mathbb H}}

\def\1{\sqrt{-1}\:}

\newcommand{\cntrct}                
{\hspace{2pt}\raisebox{1pt}{\text{$\lrcorner$}}\hspace{2pt}}

\renewcommand{\bar}{\overline}
\renewcommand{\phi}{\varphi}
\renewcommand{\epsilon}{\varepsilon}
\renewcommand{\geq}{\geqslant}
\renewcommand{\leq}{\leqslant}

\renewcommand{\Im}{\operatorname{Im}}

\newtheorem{theorem}{Theorem}[section]
\newtheorem{proposition}[theorem]{Proposition}
\newtheorem{lemma}[theorem]{Lemma}
\newtheorem{corollary}[theorem]{Corollary}
\newtheorem{remark}[theorem]{Remark}

\newtheorem{definition}[theorem]{Definition}

\title{Hodge decomposition for Cousin groups and for  Oeljeklaus-Toma manifolds}

\author[Alexandra Otiman]{Alexandra Otiman}

\address{Alexandra Otiman,
Roma Tre University, Department of Mathematics and Physics, Largo San Leonardo Murialdo,
Rome, Italy
 AND\newline 
Institute of Mathematics ``Simion Stoilow" of the Romanian Academy, 21,
Calea Grivitei, 010702, Bucharest, Romania, AND\newline 
University of Bucharest, Research Center in Geometry, Topology and Algebra, Faculty of Mathematics and Computer Science, 14 Academiei Str., Bucharest, Romania}

\email{aiotiman@mat.uniroma3.it, alexandra.otiman@imar.ro}

\author[Matei Toma]{Matei Toma}
\address{Matei Toma, 
Universit\'e de Lorraine, CNRS, IECL, F-54000 Nancy, France}

\email{Matei.Toma@univ-lorraine.fr}
\urladdr{\href{http://www.iecl.univ-lorraine.fr/~Matei.Toma/}{http://www.iecl.univ-lorraine.fr/~Matei.Toma/}}

\date{\today}


\begin{document}

\maketitle

\begin{abstract}
We compute the Dolbeault cohomology of  geodesically convex   domains contained in Cousin groups which satisfy a strong dispersiveness condition. As a consequence we obtain a description of the Dolbeault cohomology of Oeljeklaus-Toma manifolds and in particular the fact that the Hodge decomposition holds for their cohomology. 
\end{abstract}

\section{Introduction}

A Cousin group $X$ is a quotient $\C^n/ \Lambda$, where $\Lambda$ is a discrete subgroup of rank $n+m$, with $1 \leq m \leq n$, such that the global holomorphic functions on $X$ are  constant. They are named after P. Cousin and introduced in \cite{Cousin}.
In \cite{v83} it is shown that a Cousin group has finite dimensional Dolbeault cohomology groups provided if and only if  the discrete subgroup $\Lambda$ satisfies a certain dispersiveness  condition, which we shall describe in the paper and call {\it weak dispersiveness.} 
Moreover, Hodge decomposition is proven by Vogt to hold on $X$ under this same condition.

The aim of our paper is twofold. Firstly, we extend the ``if'' direction of Vogt's result  to open sets $U$ in $\C^n/ \Lambda$, whose inverse image in $\C^n$ are convex domains, see Theorem~\ref{main}. 
For this we need to impose a new condition on the discrete subgroup $\Lambda$, which we shall call {\it strong dispersiveness}. We show that this condition is actually equivalent to the finite generation of the Dolbeault cohomology of such domains, see Theorem~\ref{thm:reciprocal}, thus extending the ``only if'' direction of the cited result as well. 
Secondly, we use the aforementioned extension to show the Hodge decomposition  and to compute the Dolbeault cohomology of Oeljeklaus-Toma (OT) manifolds, see  Theorem~\ref{Hodgedecomposition}. These are compact complex manifolds associated to number fields allowing a positive number of real embeddings as well as a positive number of complex (non-real) embeddings, see Section \ref{section:OT}. Their construction and first properties are described in \cite{OT}.  Thus OT manifolds give  examples in any dimension of compact complex non-K\" ahler manifolds for which Hodge decomposition holds, or equivalently, the Fr\" olicher spectral sequence degenerates at the first page.
As a consequence, we also obtain a new way of computing the Dolbeault cohomology of Inoue-Bombieri surfaces, which are obtained as Oeljeklaus-Toma manifolds of complex dimension 2, without using powerful tools like the Riemann-Roch theorem or Serre duality and providing instead a more complex-analytical proof. 

\noindent{{\bf Acknowledgements:}} The authors acknowledge the support of the Max Planck Institute for Mathematics in Bonn where part of this research  was carried out. A. O. is partially supported by a grant of Ministry of Research and Innovation, CNCS - UEFISCDI,
project number PN-III-P4-ID-PCE-2016-0065, within PNCDI III. M.T. warmly thanks Karl Oeljeklaus for many interesting discussions on OT-manifolds since the publication of \cite{OT}.


\section{Preliminary facts on Cousin groups}\label{section:prelim}

We present in this section basic definitions and results about Cousin groups and introduce the notions of {\it weak} and {\it strong dispersiveness}, see Definition \ref{def:dispersiveness}. 

\begin{definition}\label{def:Cousin-group} A connected complex Lie group $X$ admitting no non-constant global holomorphic functions is called a {\em Cousin group} or a {\em toroidal group}.
\end{definition}

Cousin groups of complex dimension $n$ are shown to appear as quotients  $X=\C^{n}/\Lambda$, where $\Lambda$ is a discrete subgroup of $\C^n$ of rank $n+m$, with $1 \leq m \leq n$, cf. \cite[Proposition 1.1.2]{AK}. 
Moreover $\Lambda$ may be assumed to be generated by the columns of a matrix of the form:
\begin{equation}\label{standard}
P=\begin{pmatrix} 
O_{m, n-m} & T_{m, 2m} \\ 
I_{n-m} & R_{n-m, 2m} 
\end{pmatrix},
\end{equation}
 which we shall call the normal form,
where $I_{n-m}$ is the $n-m$ identity matrix, $T_{m, 2m}$ is a basis of the lattice of an $m$-dimensional complex torus and $R$ has real entries.  Furthermore, one can  arrange $T$ such that the normal form is:
\begin{equation}\label{standard2}
P=\begin{pmatrix} 
O_{m, n-m} & I_{m} & M+\mathrm{i} N  \\ 
I_{n-m} & R_1 & R_2
\end{pmatrix},
\end{equation}
where $M$ and $N$ have real entries and $N$ is invertible, see \cite[Proposition 2]{v82}, \cite[Proposition 1]{v83}. In the above situation we will say that $P$  is the {\em  period matrix} of $\Lambda$.

\begin{proposition}[{\cite[Proposition 2]{v82}}]\label{prop:Cousin} Suppose that $X=\C^n/\Lambda$ with $\Lambda$ generated by the columns of a matrix $P$ in normal form \eqref{standard}. Then 
$X$ is a Cousin group if and only if for any $\sigma \in \Z^{n-m}\setminus \{0\}$, $^{t}\sigma R \not \in \Z^{2m}$.
\end{proposition}

\begin{proposition}\label{prop:global-functions}
Let $X=\C^n/\Lambda$ be a Cousin group and let $U\subset X$ be a non-empty open subset whose inverse image $\tilde U$ in $\C^n$ is convex. Then  any global holomorphic function on $U$ is constant.
\end{proposition}
\begin{proof} We use essentially that $\C^n/ \Lambda$ is a Cousin group, a similar argument as in \cite[Lemma 2.4]{OT} and the fact that $\tilde{U}$ is convex.
We may and will assume that $\Lambda$ is generated by the columns of a matrix $P$ in normal form \eqref{standard}.

For $(z^0_1, \ldots, z^0_n) \in \tilde{U}$,  the set  $conv((z^0_1, \ldots, z^0_n) + \Lambda)$ is a real affine $(n+m)$-dimensional plane in $\C^n$, where by $conv(S)$ we mean the convex hull of $S$. It is also a subset of 
$ \tilde{U}$ by the convexity and $\Lambda$-invariance of $\tilde{U}$.  Since the functions $\mathrm{Im} \, z_{m+1}, \ldots,\mathrm{Im} \, z_{n}$ are $\Lambda$-invariant we get $conv((z^0_1, \ldots, z^0_n) + \Lambda)=\C^m \times ((z^0_{m+1}, \ldots, z^0_n)+\R^{n-m})$. Therefore
\begin{equation}
\tilde{U} = \bigcup_{(z^0_1, \ldots, z^0_n) \in \tilde{U}} \C^m \times ((z^0_{m+1}, \ldots, z^0_n)+\R^{n-m})=\C^m \times \bigcup_{(z^0_1, \ldots, z^0_n) \in \tilde{U}} ((z^0_{m+1}, \ldots, z^0_n)+\R^{n-m}).
\end{equation} 
Thus, $\tilde{U}=\C^m \times W$, where $W \subset \C^{n-m}$ is a convex domain, hence Stein, and moreover $\Z^{n-m}$-invariant.

Let now $f$ be a holomorphic function on $U$, $\tilde{f}$ its lift to $\tilde{U}$ and choose arbitrarily $w \in W$. Since $\C^m \times (w+ \R^{n-m})/ \Lambda$ is diffeomorphic to $(S^1)^{n+m}$, $\tilde{f}$ is bounded on $\C^m \times (w+ \R^{n-m})$ and therefore constant on $\C^m \times \{w\}$. Using the fact that $\C^n/\Lambda$ is a Cousin group and   Proposition~\ref{prop:Cousin} we get $^{t}\sigma R \notin \Z^{2m}$ for all $\sigma \in \Z^{n-m}\setminus \{0\}$, hence the group generated by the column vectors $(I_{n-m} \, \, R)$ is dense in $\R^{n-m}$. Consequently, the image of $\C^m \times \{w\}$ is dense in $\C^m \times (w+ \R^{n-m})/\Lambda$ and thus, $f$ is constant on $\C^m \times (w+ \R^{n-m})/\Lambda$ and $\tilde{f}$ is constant on $\C^m \times (w+ \R^{n-m})$. By the identity principle, $\tilde{f}$ has to be constant on $\tilde{U}$.
\end{proof}

We now introduce two notions of {\em dispersiveness} which will play an important role in this paper.

\begin{definition}\label{def:dispersiveness}
A discrete subgroup $\Lambda$ in normal form \eqref{standard2} is said to be {\em strongly dispersive}, (respectively {\em weakly dispersive}) if 
$$ \forall a \in (0, 1), \ (respectively \ \exists a \in (0, 1)), \ \exists C(a)>0, \ \forall  \sigma \in \Z^{n-m}\setminus \{0\}, \ \forall \tau \in \Z^{2m}$$
\begin{equation}\label{latice}
|| ^{t}\sigma R + ^{t}\tau|| \geq C(a)a^{|\sigma|}.
\end{equation}
 
\end{definition}

In \cite{v82} the following example of a discrete subgroup $\Lambda_\alpha$ is considered with period basis $P_\alpha$ in normal form: 
\begin{equation}\label{eq:example}
P_\alpha=\begin{pmatrix} 
0 & 1 & \mathrm{i} \\ 
1 & \alpha & 0
\end{pmatrix},
\end{equation}

where $\alpha$ is a real number. By  Proposition~\ref{prop:Cousin}  
  $\CC^2/\Lambda_\alpha$ is a Cousin group if and only if $\alpha$ is irrational.

Vogt shows in \cite{v82} that for 
$\alpha=\sum_{j=1}^\infty \frac{1}{10^{10^{j!}}}$ the discrete subgroup 
$\Lambda_\alpha$ is not weakly dispersive.

\begin{remark}\label{rem:strong-and-weak} 
Set $u_0:=1$, $u_{j+1}:=10^{u_j}$ for all $j\in \NN$, and
$$\alpha:=\sum_{j=1}^\infty \frac{1}{u_j}.$$ Then the discrete subgroup $\Lambda_\alpha$ generated by the columns of the matrix $P_\alpha$ given by \eqref{eq:example} is
  weakly dispersive but not strongly dispersive.
\end{remark}
\begin{proof} 
The strong (respectively weak) dispersiveness condition for $\Lambda_\alpha$ is rephrased as 
$$ \forall a \in (0, 1), \ (respectively \ \exists a \in (0, 1)), \ \exists C(a)>0, \ \forall  q \in \Z\setminus \{0\}, \ \forall p \in \Z$$
\begin{equation}\label{eq:dispersiveness_for_alpha}
| q\alpha-p | \geq C(a)a^{|q|}.
\end{equation}

For $q=u_k$, $k\ge 1$, we get
$$\inf_{p\in\Z}| q\alpha-p |=\sum_{j=k+1}^\infty \frac{u_k}{u_j}<\frac{2u_k}{u_{k+1}}<\frac{2^{u_k}}{u_{k+1}}=\frac{1}{5^q},$$
hence $\Lambda_\alpha$ cannot be strongly dispersive for our choice of $\alpha$.

We now check the weak dispersiveness of $\Lambda_\alpha$. For a real number $\beta$ we denote by $\{\beta\}$ its fractional part. Then for $u_k\le q<u_{k+1}$, $k>0$ we get
$$\{ q\alpha\}>\frac{1}{10^{u_k}}\ge \frac{1}{10^q}.$$ It remains to estimate $1-\{ q\alpha\}$. 
But it is clear that the $u_k+1$-st decimal digit of $\{ q\alpha\}$ is $0$, hence the $u_k+1$-st decimal digit of $1-\{ q\alpha\}$ is $9$. Thus 

$$1-\{ q\alpha\}\ge\frac{9}{10}\frac{1}{10^{u_k}}\ge\frac{9}{10} \frac{1}{10^q},$$
which proves weak dispersiveness of $\Lambda_\alpha$ by taking $a=\frac{1}{10}$.
\end{proof}

Examples of strongly dispersive discrete subgroups are provided by the following

\begin{proposition}\label{prop:algebraic_lattices}
If $\Lambda$  is a discrete subgroup defining a Cousin group and such that all the entries of some period matrix are algebraic numbers, then $\Lambda$ is strongly dispersive.
\end{proposition}
\begin{proof}
By using a generalization of Liouville's Theorem on the approximation of algebraic numbers (\cite[Theorem 1.5]{Enc}) it is proved in  \cite[Theorem 4.3]{bo} that the discrete subgroup $\mathcal{O}_K$ is weakly dispersive, see Section \ref{section:OT} for notations.

More precisely in the proof of \cite[Theorem 4.3]{bo} it is shown that if $R$ is a $k \times l$ matrix with elements algebraic numbers, then there exist constants $C>0$ and $A < 0$ such that for any $\sigma \in \Z^{k} \setminus \{0\}$ and every $\tau \in \Z^{l}$, $||^{t}\sigma R + ^{t}\tau|| \geq C |\sigma|^A$. But this condition is  stronger than strong dispersiveness, since clearly for any $a \in (0, 1)$, there exists a  constant $C(a)$ such that $|\sigma|^A \geq C(a) a^{|\sigma|}$ for all  $\sigma \in \Z^{k} \setminus \{0\}$. 
\end{proof}

The following result proved by Chr. Vogt in \cite{v82}, \cite{v83} will be extended in Section \ref{section:Cousin} to the case of  open sets  in $\C^n/ \Lambda$, whose inverse image in $\C^n$ are convex domains. 

\begin{theorem}[{\cite{v82},\cite{v83}}]\label{thm:vogt}
If $X= \C^n/ \Lambda$ is a Cousin group, then $H^1(X, \mathcal{O})$ is finite dimensional if and only if 
$\Lambda$ is weakly dispersive. Moreover, in this situation all the Dolbeault cohomology groups $H^{p, q}_{\overline{\partial}}(X)$ are finite dimensional and $X$ satisfies the Hodge decomposition.

\end{theorem}
Additionally, Vogt gives several equivalent conditions for the finite dimensionality of $H^1(X, \mathcal{O})$ in terms of the discrete subgroup $\Lambda$, the holomorphic line bundles on $X$ and the generators of $H^1(X, \mathcal{O})$. 


\section{Dolbeault cohomology of convex domains in Cousin groups. 
}\label{section:Cousin}

In this section we will prove analogous results to those of Theorem \ref{thm:vogt} for open subsets $U$ in Cousin groups $X=\C^n/\Lambda$, whose inverse images 
in $\C^n$ 
are convex domains. 
We will call such open sets $U$ in $X$ simply {\it  convex} since they are the geodesically convex open subsets of the Lie group $X$ for the unique system of geodesics which are left and right invariant; these are given by translates of one-parameter subgroups of $X$. In particular, the definition of a  convex open subset in a Cousin group $X$ does not depend on the chosen presentation $\C^n/\Lambda$ for $X$; see also \cite[Proposition 1.1.8]{AK} for an alternative argument.


\begin{theorem}\label{main}
Let $U$ be a  domain of a Cousin group $X=\C^n/\Lambda$, whose inverse image $\tilde U$ in $\C^n$ 
is a convex domain.  
 If $\Lambda$ is strongly dispersive then $H^q(U, \Omega^{p})$ is finitely generated and moreover, 
$$\{[dz_I \wedge d\overline{z}_J] \mid I \subseteq \{1, \ldots, n\}, J \subseteq \{1, \ldots, m\}, |I|=p, |J|=q\}$$ 
is a basis and thus, $\mathrm{dim}_{\C} H^q(U, \Omega^{p})={n \choose p} \cdot {m \choose q}$.
\end{theorem}

We follow the lines of the proofs of Proposition 4 and Proposition 5 in \cite{v83} and adapt them to the new setting. 

We start with  the following lemma:

\begin{lemma}\label{holodep} Let $q\geq 1$. Any $\Lambda$-invariant $\bar\partial$-closed $(0, q)$-form $\omega$ on $\tilde{U}$ is $\overline{\partial}$-cohomologous to a $\Lambda$-invariant $(0, q)$-form on $\tilde{U}$, whose coefficients depend holomorphically on $z_{m+1}, \ldots, z_n$.
\end{lemma}

The proof of this Lemma follows the same steps as in \cite[Proposition 4]{v83}. The only new thing we have to check is that 
 $U$ is the total space of a locally trivial holomorphic fibration over a complex torus with fibre a Stein manifold. To this aim we remark as in the proof of  Proposition~\ref{prop:global-functions} that $\tilde{U}=\C^m \times W$, where $W$ is a convex $\Z^{n-m}$-invariant domain.
Since $U=\tilde{U}/\Lambda$, the map
\begin{equation}\label{fibrare}
\pi: U \rightarrow \C^m/T,
\end{equation}
given by $\pi([z_1, \ldots, z_n])=\widehat{(z_1, \ldots, z_m)}$, is well-defined, where $[\cdot]$ and $\hat{\cdot}$ are classes with respect to taking quotients by $\Lambda$ and by the lattice generated by the columns of $T$, respectively, cf. \cite[Proposition 2]{v83}. 
Clearly, $\pi$ is a holomorphic map and in fact, \eqref{fibrare} is a fibration with fibre  isomorphic to $F: = W/\Z^{n-m}$. 
Via the map $exp(2\pi \mathrm{i} \, \cdot)$, $\C^{n-m}/\Z^{n-m}$ is biholomorphic to $(\C^*)^{n-m}$, so we regard $F$ directly as an open subset of $(\C^*)^{n-m}$. In fact, seen in this way, $F$ is a (relatively complete) logarithmically convex Reinhardt domain in $(\C^*)^{n-m}$ and is therefore a domain of holomorphy in $\C^{n-m}$ and thus Stein,  \cite[Lemma 1.7]{Car} (see also \cite[Theorem 1.11.13]{JP}). 

\hfill

We can now proceed to the proof of Theorem~\ref{main}.
Unlike the situation of \cite{v83} we need to deal with domains of holomorphy different from $(\C^*)^{n-m}$. This is no major obstruction in the case of the proof of the previous Lemma, but will entail substantial modifications in the proof of Theorem \ref{main} as compared to  \cite[Proposition 5]{v83}. 
More precisely, in loc. cit. essential use is made of the classical fact that a Laurent series $\sum_{\sigma\in\Z^p} a_\sigma z^\sigma$ is convergent on $(\C^*)^{p}$ if and only if  $\lim\sup_{\sigma\in\Z^p}\sqrt[|\sigma|]{|a_\sigma|}=0$. In our set-up this is no longer applicable and we need to show that we get positive convergence radii of corresponding Laurent series around each point of our domain of holomorphy. For this the strong dispersiveness condition will be used in a crucial way.

\begin{proof}
We divide the proof in two steps. 

{\it Step 1:} We will show that any $\overline{\partial}$-closed, $\Lambda$-periodic $(p, q)$-form $\omega$ on $\tilde{U}$ is $\overline{\partial}$-cohomologous to a form $\sum_{I, J} c_{I J} dz_I \wedge d\overline{z}_J$ with constant coefficients $c_{IJ} \in \C$.

If $\omega$ is a $(p, 0)$-form, the statement is obvious, as $\omega$ has to be of type $\sum_I f_I dz_I$, with $f_I$ holomorphic $\Lambda$-invariant functions on $\tilde{U}$, but these are  constant by  Proposition~\ref{prop:global-functions}.

Let now $q \geq 1$. Once we prove the statement for $(0, q)$-forms, it will immediately follow for $(p, q)$-forms as well, since $\omega=\sum_{I,J} f_{IJ} dz_I \wedge d\overline{z}_J=\sum dz_I \wedge (\sum_J f_{IJ} d\overline{z}_J)$ and each $\sum_{J}f_{IJ} d\overline{z}_J$ is $\overline{\partial}$-closed.

Therefore, take $\omega=\sum_J f_J d\overline{z}_J$ on $\tilde{U}$, $\overline{\partial}$-closed, $\Lambda$-periodic. By Lemma~\ref{holodep}, we may assume that $f_J$ depend holomorphically on $z_{m+1}, \ldots, z_n$ and $J \subseteq \{1, 2, \ldots, m\}$.

The strategy will be, as in the proof of \cite[Proposition 5]{v83}, to define a $\Lambda$-invariant $(0, q-1)$-form $\eta$ on $\tilde{U}$, using the Fourier expansion of the $\Lambda$-invariant functions $f_J$ and show there exist $c_J \in \C$ such that $\omega-\sum_J c_J d\overline{z}_J = \overline{\partial}\eta$. Formally, $\eta$ has the same expression as in the proof of \cite[Proposition 5]{v83}, but the difficult part in our case will be to show that $\eta$ is well defined, namely, that its coefficients are convergent series. 

Let $(z_1, \ldots, z_m) =: x+ \mathrm{i} y$ and $(z_{m+1}, \ldots, z_n)=: w$. For any $\pi, \rho \in \Z^m$ and $\sigma \in \Z^{n-m}$, we define the following function on $\tilde{U}$:
\begin{equation}
\gamma_{\pi, \rho, \sigma}(z_1, \ldots, z_n)
:= exp \left( 2 \pi \mathrm{i} \cdot \left( (^{t}\pi - ^{t}\sigma R_1)x + 
(^{t} \rho-^{t}\pi M+ 
^{t}\sigma(R_1 M-R_2))\cdot N^{-1}y+^{t}\sigma w\right)\right)
\end{equation}
Each $\gamma_{\pi, \rho, \sigma}$ is $\Lambda$-invariant and thus, we develop $f_J$ in Fourier series on $\tilde{U}$:
\begin{equation}
f_J=\sum_{\pi, \rho, \sigma} f_{J, \pi, \rho, \sigma} \gamma_{\pi, \rho, \sigma},
\end{equation}
where $f_{J, \pi, \rho, \sigma} \in \C$.
We define the following: 
\begin{equation}
a_{\pi, \rho, \sigma}:=\tfrac{1}{2} \left( (^{t}\pi - ^{t}\sigma R_1) + \mathrm{i} (^{t}\rho-^{t}\pi M + ^{t}\sigma (R_1 M -R_2))N^{-1} \right) \in \C^m
\end{equation}

\begin{equation}
B := \{\sum_{j=1}^m b_j d\overline{z}_j 
\mid b_j \in \C\} 
\end{equation}

\begin{equation}
\lambda_{\pi, \rho, \sigma}: B \rightarrow \C, \qquad \lambda_{\pi, \rho, \sigma} (\sum_{j=1}^{m}b_j \cdot d\overline{z}_j):=\frac{\sum_{j=1}^m b_j a_{\pi, \rho, \sigma, j}}{2 \pi \mathrm{i} ||a_{\pi, \rho, \sigma}||^2},
\end{equation}
where $a_{\pi, \rho, \sigma, j}$ is the $j$-th component of $a_{\pi, \rho, \sigma}$. We will see at a further point in the proof that $||a_{\pi, \rho, \sigma}||$ does not vanish.

We extend $\lambda_{\pi, \rho, \sigma}$ to a homomorphism:
\begin{equation}
\lambda_{\pi, \rho, \sigma} \rfloor : \Lambda^kB \rightarrow \Lambda^{k-1}B \qquad
\lambda_{\pi, \rho, \sigma} \rfloor (\alpha_1 \wedge \ldots \wedge \alpha_k)=\sum_{p=1}^k (-1)^{k-p} \lambda_{\pi, \rho, \sigma}(\alpha_p) \alpha_1 \wedge \ldots \wedge \hat{\alpha}_p \wedge \ldots \wedge \alpha_k
\end{equation}
and define the $\Lambda$-periodic $(0, q-1)$-form on $\tilde{U}$.
\begin{equation}\label{eta}
\eta=\sum_{(\pi, \rho, \sigma) \neq 0} (-1)^{q-1} \left( \lambda_{\pi, \rho, \sigma} \rfloor (\sum_J f_{J, \pi, \rho, \sigma} d\overline{z}_J)\right) \gamma_{\pi, \rho, \sigma}.
\end{equation}
By a straightforward, but lengthy, computation, presented in \cite{v83}, one gets that $\overline{\partial} \eta = \omega - \sum_J f_{J, 0, 0, 0}d\overline{z}_J$.
Step 1 will be clear once we show that  $\eta$ is a convergent series. It is at this point that 
the strong dispersiveness of $\Lambda$ will play an essential role.

We now proceed to the proof of the convergence of the series $\eta$. We show first by using \eqref{latice} that for any $a \in (0, 1)$, there exists a constant $C_1(a)>0$ such that for any $(\pi, \rho, \sigma) \neq 0$, $||a_{\pi, \rho, \sigma}|| \geq C_1(a) a^{|\sigma|}$.

Take $k_1 :=||MN^{-1}||$, $k_2 := \frac{1}{||N||}$.

Then clearly
\begin{equation}\label{k1}
||\alpha MN^{-1}|| \leq k_1 ||\alpha||, \qquad \forall \alpha \in \R^{m}
\end{equation}
\begin{equation}\label{k2}
 ||\alpha N^{-1}|| \geq k_2||\alpha||, \qquad \forall\alpha \in \R^m
\end{equation}

Let $k:=\frac{k_2}{1+k_1+k_2}$. If $(\pi, \rho, \sigma) \neq 0$ is such that $2 ||\mathrm{Re}(a_{\pi, \rho, \sigma})||=||^{t}\pi-^{t}\sigma R_1||\leq kC(a)a^{|\sigma|}$, then by \eqref{latice}, we get $||^{t}\rho-^{t}\sigma R_2|| \geq (1-k)C(a) a^{|\sigma|}$ and therefore by \eqref{k1} and \eqref{k2}:

\begin{equation}
2||\mathrm{Im}(a_{\pi, \rho, \sigma})||=||(^{t}\rho - ^{t}\sigma R_2)N^{-1} - (^{t}\pi - ^{t}\sigma R_1)MN^{-1}|| \geq (k_2(1-k)C(a) -k_1kC(a)) a^{|\sigma|}=kC(a)a^{|\sigma|}.
\end{equation}

This means that for $C_1(a):=\frac{1}{2}kC(a)$ we get
\begin{equation}\label{ineqA}
||a_{\pi, \rho, \sigma}|| \geq C_1(a) a^{|\sigma|}, \forall (\pi, \rho, \sigma) \neq 0.
\end{equation}
In particular  $||a_{\pi, \rho, \sigma}||$ does not vanish for $(\pi, \rho, \sigma) \neq 0$.

By the expression in \eqref{eta}, we deduce that:
\begin{equation}
\eta=\sum_{|K|=q-1} \left( \sum_{(\pi, \rho, \sigma) \neq 0} t_{\pi, \rho, \sigma}\gamma_{\pi, \rho, \sigma} \right)d\overline{z}_K,
\end{equation}
where $t_{\pi, \rho, \sigma}$ is a finite sum of terms of type $\pm f_{K \cup \{j\}, \pi, \rho, \sigma}\frac{a_{\pi, \rho, \sigma, j}}{||a_{\pi, \rho, \sigma}||^2}$. We need to show that $h_K:=\sum_{(\pi, \rho, \sigma) \neq 0} t_{\pi, \rho, \sigma}\gamma_{\pi, \rho, \sigma}$ is convergent on $\tilde{U}$.

We prove that for each $|J|=q$, $h_J:=\sum_{(\pi, \rho, \sigma) \neq 0} f_{J, \pi, \rho, \sigma} \frac{a_{\pi, \rho, \sigma, j}}{||a_{\pi, \rho, \sigma}||^2}\gamma_{\pi, \rho, \sigma}$ is convergent on $\tilde{U}$ and this will suffice.

Fix $z_0=x_0 + \mathrm{i}y_0 \in \C^m$. Then 
$$h_J(z_0, w)=\sum_{\sigma} \left( \sum_{\pi, \rho} f_{J, \pi, \rho, \sigma} \frac{a_{\pi, \rho, \sigma, j}}{||a_{\pi, \rho, \sigma}||^2}exp(2 \pi \mathrm{i} a(z_0)) \right)exp(2 \pi \mathrm{i} ^{t}\sigma \cdot w),$$
where $a(z):= (^{t}\pi - ^{t}\sigma R_1)x + (^{t}\rho-^{t}\pi M+ ^{t}\sigma(R_1 M-R_2))\cdot N^{-1}y \in \R$ , for any $z \in {\C^m}$. 

Recall that $\tilde{U}=\C^m \times W$, where $W$ is a $\Z^{n-m}$-invariant convex Stein domain in $\C^{n-m}$.
We prove now that $\sum_{\sigma}\left(\sum_{\rho, \pi} f_{J, \pi, \rho, \sigma} \frac{a_{\pi, \rho, \sigma, j}}{||a_{\pi, \rho, \sigma}||^2}exp\left(2\pi \mathrm{i} (a(z))\right)\right)w^{\sigma}$ is uniformly absolutely convergent on $\C^m \times W_1$, where $W_1=exp(2\pi \mathrm{i} W)$. 
Note that $W_1$ is a Reinhardt domain, therefore, $W_1=\mathbb{T}^{n-m} \cdot S$, where $S=\{(|w_1|, \ldots, |w_{n-m}|) \mid (w_1, \ldots, w_{n-m}) \in W_1\}$.

Choose $w^0=(w^0_1, \ldots, w^0_{n-m}) \in W_1\cap\RR_+^{n-m}$, a neighbourhood $U_{z_0}$ of $z_0$ in $\C^m$ and $S^{w^0}_{\epsilon} = (w^0_1-\epsilon, w^0_1 + \epsilon) \times  \ldots \times (w^0_{n-m}-\epsilon, w^0_{n-m} + \epsilon)$, for a small $\epsilon >0$, such that $\bar S^{w^0}_{\epsilon}\subset W_1$.

For any $a \in (0, 1)$, on $U_{z_0} \times \mathbb{T}^{n-m} \cdot S^{w^0}_{\tfrac{\epsilon}{2}}$, we have by \eqref{ineqA}:

\begin{dmath}
\sum_{\sigma}\left(\sum_{\rho, \pi} |f_{J, \pi, \rho, \sigma}| \cdot|\frac{a_{\pi, \rho, \sigma, j}}{||a_{\pi, \rho, \sigma}||^2}| \cdot |exp\left(2\pi \mathrm{i} (a(z))| \right)\right)|w^{\sigma}| = \sum_{\sigma}\left(\sum_{\rho, \pi} |f_{J, \pi, \rho, \sigma}| \cdot |\frac{a_{\pi, \rho, \sigma, j}}{||a_{\pi, \rho, \sigma}||^2}|\right)|w^{\sigma}| \leq \\
mC_1(a)^{-1}\sum_{\sigma}\left(\sum_{\rho, \pi} |f_{J, \pi, \rho, \sigma}| \cdot a^{-|\sigma|} \right)|w^{\sigma}|.
\end{dmath}

We split now the series $\sum_{\sigma}\left(\sum_{\rho, \pi} |f_{J, \pi, \rho, \sigma}| \cdot a^{-|\sigma|} \right)|w^{\sigma}|$ in a sum of $2^{n-m}$ series
 
\begin{equation}
h^a_g:=\sum_{\sigma \in D_g} \left(\sum_{\rho, \pi} |f_{J, \pi, \rho, \sigma}| \cdot a^{-|\sigma|} \right)|w^{\sigma}|,
\end{equation}
where $g: \{1, \ldots, n-m\} \rightarrow \{-1, 1\}$ and $D_g= \{(\sigma_1, \ldots, \sigma_{n-m}) \in \Z^{n-m} \setminus \{\noindent{\bf{0}}\} \mid \mathrm{sgn}(\sigma_i)=g(i)\}$. 
By convention we consider $\mathrm{sgn}(0)=1$.
Then on $U_{z_0} \times \mathbb{T}^{n-m} \cdot S^{w^0}_{\frac{\epsilon}{2}}$:

\begin{align*}
h^a_g  = \\
\sum_{\sigma \in D_g, \rho, \pi} |f_{J, \pi, \rho, \sigma}| \cdot a^{-|\sigma|} |\tfrac{w_1}{w^0_1+g(1)\epsilon}|^{\sigma_1}\ldots|\tfrac{w_{n-m}}{w^0_{n-m}+g(n-m)\epsilon}|^{\sigma_{n-m}} |w^0_1+g(1)\epsilon|^{\sigma_1}\ldots |w_n^0 + g(n-m)\epsilon|^{\sigma_{n-m}}\\
 \leq \sum_{\sigma \in D_g, \rho, \pi} |f_{J, \pi, \rho, \sigma}| \cdot a^{-|\sigma|} \delta^{|\sigma|}_{(w_1^0, \ldots, w^0_{n-m}), \epsilon} |w^0_1+g(1)\epsilon|^{\sigma_1}\ldots |w_n^0 + g(n-m)\epsilon|^{\sigma_{n-m}}, 
\end{align*}
where $\delta_{(w_1^0, \ldots, w^0_{n-m}), \epsilon}= max\{|\tfrac{w^0_1+g(j)\tfrac{\epsilon}{2}}{w^0_1+g(j)\epsilon}|^{\mathrm{sgn}(\sigma_j)} \qquad \mid \sigma_{j} \neq 0\} < 1$.

We can choose now $a$ to be $\delta_{(w_1^0, \ldots, w^0_{n-m}), \epsilon}$ and thus, 
\begin{equation*}
h^{\delta}_g \leq \sum_{\sigma \in D_g, \rho, \pi} |f_{J, \pi, \rho, \sigma}| \cdot |w^0_1+g(1)\epsilon|^{\sigma_1}\ldots |w_n^0 + g(n-m)\epsilon|^{\sigma_{n-m}}.
\end{equation*}

But the series in the right hand side above is bounded by a constant $C((w_1^0, \ldots, w^0_{n-m}), \epsilon)$, since $f_J$ is holomorphic in $z_{m+1}, \ldots, z_n$ and thus, the series $\sum_{\sigma, \pi, \rho} f_{J, \pi, \rho,\sigma} exp(2 \pi \mathrm{i} z)w^{\sigma}$ is absolutely uniformly convergent on $U_{z_0} \times \mathbb{T}^{n-m} \cdot S^{w^0}_{\frac{\epsilon}{2}}$.

What we actually proved above is that $\sum_{\sigma \neq 0}\sum_{(\pi, \rho)} f_{J, \pi, \rho, \sigma} \frac{a_{\pi, \rho, \sigma, j}}{||a_{\pi, \rho, \sigma}||^2}\gamma_{\pi, \rho, \sigma}$ is convergent. But if $\sigma=0$, we observe that for $(\pi, \rho) \neq 0$, $||a_{\pi, \rho, 0}|| \geq 1$. Indeed, it is clear by $2||\mathrm{Re}\, a_{\pi, \rho, 0}||=||^{t}\pi||$ and $2||\mathrm{Im}\, a_{\pi, \rho, 0}||=||^{t}\rho-^{t}\pi||$. 

Consequently, the missing part of $h_J$, which is $\sum_{(\pi, \rho) \neq 0} f_{J, \pi, \rho, 0} \frac{a_{\pi, \rho, 0, j}}{||a_{\pi, \rho, 0}||^2}\gamma_{\pi, \rho, 0}$, is dominated by   $\sum_{(\pi, \rho) \neq 0} f_{J, \pi, \rho, 0} \gamma_{\pi, \rho, 0}$, which is convergent since $f_J$ is. We conclude that $\eta$ is convergent on $\tilde{U}$ and Step 1 is proved.

\hfill

{\it Step 2.} We prove now that $\{[dz_{I} \wedge d\overline{z}_J] \mid I=(1 \leq i_1 < \ldots < i_p \leq n), J=(1 \leq i_1 < \ldots < i_q \leq m)\}$ is a basis for $H^{q}(U, \Omega^{p})$. Step 1 tells us that $H^{q}(U, \Omega^{p})$ is generated by $\{[dz_{I} \wedge d\overline{z}_J] \mid I=(1 \leq i_1 < \ldots < i_p \leq n), J=(1 \leq i_1 < \ldots < i_q \leq m)\}$, therefore $\mathrm{dim}_{\C}H^{q}(U, \Omega^p) \leq {n \choose p}{m \choose q}$.
Since all $H^q (U, \Omega^p)$ and $H^*_{dR}(U, \C)$ are finitely generated, we can apply Fr\" olicher's inequality and get:
\begin{equation}\label{frolicher}
\mathrm{dim}_{\C} H_{dR}^l(U, \C) \leq \sum_{p+q=l}\mathrm{dim}_{\C} H^q(U, \Omega^p) \leq \sum_{p+q=l} {n \choose p}{m \choose q}={n+m \choose l}
\end{equation}
As $U \simeq (S^1)^{n+m} \times \R^{n-m}$, $\mathrm{dim}_{\C} H_{dR}^l(U, \C)={n+m \choose l}$, therefore we have equality in \eqref{frolicher} and the conclusion follows. 
\end{proof}

The fact that equality holds in \eqref{frolicher} immediately implies

\begin{corollary}
If $\Lambda$ is strongly dispersive, then the Hodge decomposition holds  for any convex domain $U$ in the Cousin group $\C^n/\Lambda$. 
\end{corollary}

We next state and prove a converse of Theorem~\ref{main}:

\begin{theorem}\label{thm:reciprocal}
If $H^1(U, \mathcal{O})$ is finite dimensional for every open  
subset $U$ of the Cousin group $X=\C^n / \Lambda$ such that its inverse image $\tilde U$ in $\C^n$ is convex, then $\Lambda$ is strongly dispersive. 
\end{theorem}

\begin{proof}
We shall argue by contradiction, namely, we show that if $\Lambda$ is not strongly dispersive, then there exists an open convex $U$ such that $H^1(U, \mathcal{O})$ is infinite dimensional. Indeed, if $\Lambda$ is not strongly dispersive, 
$$\exists a \in (0, 1), \forall C>0, \exists \sigma(C) \in \Z^{n-m} \setminus\{0\}, \exists \tau(C) \in \Z^{2m}, ||^{t}\sigma(C)R - ^{t}\tau(C)|| < Ca^{|\sigma(C)|},$$
which by taking $C=\frac{1}{k}$, with $k \in \N^*$, implies that:
\begin{equation}\label{negatie}
\exists a \in (0, 1), \forall k \in \N^*, \exists \sigma(\tfrac{1}{k}) \in \Z^{n-m} \setminus\{0\}, \exists \tau(\tfrac{1}{k}) \in \Z^{2m}, ||^{t}\sigma(\tfrac{1}{k})R - ^{t}\tau(\tfrac{1}{k})|| < \frac{1}{k}a^{|\sigma(\tfrac{1}{k})|}.
\end{equation}
For convenience, we shall use the notation $\sigma(k)$ instead of $\sigma(\frac{1}{k})$. We can assume that $\sigma(k) \neq \sigma(l)$ for $k \neq l$, otherwise we can extract a subsequence $(k_i)_{i\in \N}$ such that $\sigma(k_i)$ are all different. Indeed, if we had a finite set of values for the sequence $(\sigma(k))_{k \in \N}$, we would have a subsequence $(k_j)_{j \in \N}$ such that $\sigma(k_j)=c \neq \noindent{{\bf 0}}$, for all $k_j$. Then by \eqref{negatie}, we get that $^{t}cR \in \Z^{2m}$, which is impossible by Proposition~\ref{prop:Cousin}. Moreover, by taking again a subsequence if needed, we can consider that for any position $i \in \{1, \ldots, n-m\}$, $\mathrm{sgn}(\sigma(k)_i)$ is constant. Therefore, \eqref{negatie} tells us that:
\begin{equation}\label{negatie2}
\exists a \in (0, 1), \forall k \in \N^*, \exists \sigma(k) \in \Z^{n-m} \setminus\{0\}, \exists \tau(k) \in \Z^{2m}, ||^{t}\sigma(k)R - ^{t}\tau(k)|| < \frac{1}{k}a^{|\sigma(k)|}\le a^{|\sigma(k)|},
\end{equation}
such that $\sigma(k) \neq \sigma(l)$, whenever $k \neq l$ and hence $|\sigma(k)| \xrightarrow[k \rightarrow \infty]{} \infty$. 

We are led to consider
$\tilde{U}:=\C^m \times \prod_{i=1}^{n-m}\H^{\mathrm{sgn}(\sigma(k)_i)}$ and $U:=\tilde{U}/\Lambda$, where we set $\H^{+1}:=\H$, $\H^{-1}:=-\H=\{z \in \C \mid \Im z<0\}$. Here we have set
$\mathrm{sgn}(0)=+1$
 by abuse of notation.
In fact by applying the automorphism $$z\mapsto(z_1,...,z_m,  \mathrm{sgn}(\sigma(k)_1)\cdot z_{m+1},...,\mathrm{sgn}(\sigma(k)_{n-m}) \cdot z_n)$$ of $\C^n$ we reduce ourselves to the situation where all $\sigma(k)_i$ are non-negative. In the sequel we will suppose that this is the case. Thus the considered convex  domain in $\C^n$ will be  $\tilde U=\C^m\times \H^{n-m}$. 

We are in a situation where the sheaf cohomology $H^1(U,\mathcal{O})$ may be computed as the group cohomology $H^1(\Lambda, H^0(\tilde U, \mathcal{O}))$, where $\Lambda$ on $H^0(\tilde U, \mathcal{O})$ naturally via translation on $\tilde U$, see \cite[Appendix to Section 2]{MumfordAbelianVarieties}.
Thus $$H^1(U,\mathcal{O})\cong H^1(\Lambda, H^0(\tilde U, \mathcal{O}))= Z^1(\Lambda, H^0(\tilde U, \mathcal{O}))/B^1(\Lambda, H^0(\tilde U, \mathcal{O})),$$ where  
\begin{align*}
Z^1(\Lambda, H^0(\tilde U, \mathcal{O})):= & \{ A: \Lambda \times \tilde{U} \rightarrow \C \ | \\ & A(\lambda, \cdot)\in H^0(\tilde U, \mathcal{O}) \ \forall \lambda \in \Lambda, \\ & A(\lambda_1+\lambda_2,z)=A(\lambda_1, z+\lambda_2)+A(\lambda_2, z), \ \forall \lambda_1, \lambda_2 \in \Lambda, \ \forall z\in \tilde{U}\},\\
B^1(\Lambda, H^0(\tilde U, \mathcal{O})):= & \{ A: \Lambda \times \tilde{U} \rightarrow \C \ |  \ \exists g\in H^0(\tilde U, \mathcal{O}) \\ & A(\lambda, z)=g(z+\lambda)-g(z) \ \forall \lambda \in \Lambda, \ \forall z\in \tilde{U}\}.
\end{align*}

The strategy is to define an infinite family of linearly independent elements in $H^1(\Lambda, H^0(\tilde U, \mathcal{O}))$.

For any $\sigma \in \Z^{n-m} \setminus \{0\}$ we set $\eta_{\sigma}:=\mathrm{max}_{r_j} |\mathrm{exp}(2\pi \mathrm{i} ^{t}\sigma \cdot r_j)-1|$, where $r_j$ for $j \in \{1, \ldots, 2m\}$ are the columns of $R$. 
We shall denote  by $v_i$ the columns of $P$.

For $\lambda=\sum_{j=1}^{n+m} n_jv_j\in \Lambda$, $n_j\in\Z$ we will further denote by 
$l(\lambda):=\sum
_{j=1}^{n+m} |n_j|$.

For each $x \in (0, 1)$, we define an element $A^{(x)}\in Z^1(\Lambda, H^0(\tilde U, \mathcal{O}))$ by:
\begin{align}\label{definitesumanzi}
A^{(x)}(\lambda, z) &:=  \sum_{k \in \N} a^{x|\sigma(k)|} \left( \frac{\mathrm{exp}(2\pi \mathrm{i} ^{t}\sigma(k) \cdot (\lambda_{m+1}, \ldots, \lambda_n))-1}{\eta_{\sigma(k)}} \right)\mathrm{exp}\left(2\pi \mathrm{i} ^{t}\sigma(k) \cdot (z_{m+1}, \ldots, z_n)\right).
\end{align}

Let us check the holomorphicity of $A^{(x)}(\lambda, \cdot)$ on $\tilde{U}$ for every $\lambda$,  the other condition being clearly satisfied. 
To this aim, as 
$\frac{\mathrm{exp}(2\pi \mathrm{i} ^{t}\sigma(k) \cdot (\lambda_{m+1}, \ldots, \lambda_n))-1}{\eta_{\sigma(k)}}$ is bounded by $l(\lambda)$ 
it suffices to check that the series 
$S:= \sum_{k \in \N} a^{x|\sigma(k)|}  |w_{m+1}|^{\sigma (k)_1} \cdots |w_n|^{\sigma (k)_{n-m}}$ is uniformly convergent on 
$\mathbb{D}^{n-m}$. But this is clear since $a^{x}<1$.

Note that
\begin{equation}\label{eq:periodicity}
A^{(x)}(v_i, z) =0, \ \forall z\in\tilde U, \ \forall i \in \{1, \ldots, n-m\}.
\end{equation}

Take now $s>0$ and $0<x_1< \ldots< x_s<1$. We will show that the classes of  $A^{(x_1)}, \ldots, A^{(x_s)}$ are $\C$-linearly independent in $H^1(\Lambda, H^0(\tilde U, \mathcal{O}))$.

Suppose that this is not the case. Then there exist $c_1, \ldots, c_s\in\C$, not all zero and a 
 holomorphic function $g$ on $\tilde{U}$ such that
\begin{equation}\label{functiag}
\sum_{i=1}^{s}c_i A^{(x_i)}(\lambda, z)=g(z+\lambda)-g(z).
\end{equation}
From \eqref{eq:periodicity} and  \eqref{functiag} we deduce that $g$ is $(0,\Z^{n-m})$-periodic and therefore has a Fourier series expansion
\begin{equation}\label{Fourierg}
g=\sum_{\sigma \in \Z^{n-m}\setminus \{0\}} g_{\sigma}\mathrm{exp}(2\pi\mathrm{i}^{t}\sigma \cdot (z_{m+1}, \ldots, z_n)).
\end{equation}
Using now \eqref{definitesumanzi} and plugging \eqref{Fourierg} in \eqref{functiag}, we get $g_{\sigma}=\sum_{i=1}^{s}c_i \frac{a^{x_i|\sigma(k)|}}{\eta_\sigma(k)}$ if $\sigma=\sigma(k)$  for some $k\in\N$ and
$g_{\sigma}=0$ otherwise. Therefore 
\begin{equation*}
g=\sum_{\sigma(k) \in \Z^{n-m}\setminus \{0\}} \left(\sum_{i=1}^{s}c_i \frac{a^{x_i|\sigma(k)|}}{\eta_{\sigma(k)}}\right)\mathrm{exp}(2\pi\mathrm{i}^{t}\sigma(k) \cdot (z_{m+1}, \ldots, z_n)).
\end{equation*}
Since $g$ is holomorphic, the following series is absolutely uniformly convergent on $ \mathbb{D}^{n-m}$:
\begin{equation*}
g_1:=\sum_{\sigma(k) \in \Z^{n-m}\setminus \{0\}} \left(\sum_{i=1}^{s}c_i \frac{a^{x_i|\sigma(k)|}}{\eta_{\sigma(k)}}\right)w_{m+1}^{\sigma(k)_1} \cdots w_n^{\sigma(k)_{n-m}}.
\end{equation*}

A straightforward computation shows that $||^{t}\sigma(k)R - ^{t}\tau(k)||<a^{|\sigma(k)|}$ for some $\tau(k)\in \Z^{2m}$ entails $\eta_{\sigma}<2 \pi a^{|\sigma(k)|}$, for all $k$.
It follows that
$$L:=\mathrm{lim}\, \mathrm{sup}_{k}\frac{1}{\sqrt[|\sigma(k)|]{\eta_{\sigma(k)}}}\ge a^{-1}.$$
Set $\rho:=\frac{1}{L}$. We have $\rho\le a$.

Define 
$$S_i:=c_i \sum_{k\in \N} \left(\frac{a^{x_i|\sigma(k)|}}{\eta_{\sigma(k)}}\right)w_{m+1}^{\sigma(k)_1} \cdots w_n^{\sigma(k)_{n-m}}.$$
We may suppose that all coefficients $c_i$ are non-zero. Restricting $g$ and the series $S_i$ to $\mathbb{D}$ via the diagonal embedding $\mathbb{D}\hookrightarrow \mathbb{D}^{n-m}$ we get 
\begin{equation}\label{eq:series}
S_1|_{\mathbb{D}}=g_1|_{\mathbb{D}}-S_2|_{\mathbb{D}}-\ldots-S_{s}|_{\mathbb{D}}.
\end{equation} 
But the convergence radius of each $S_i|_{\mathbb{D}}$ equals $a^{-x_i}\rho$ and is thus lower or equal to $a^{1-x_i}$ and also lower than $1$. It follows that the convergence radius of the series appearing on the left hand side of  equation \eqref{eq:series}  is strictly smaller than the convergence radius of the right hand side. This is a contradiction. 

In fact the family $A^{(x)}$, $x\in(0,1)$  provides an infinite set of linearly independent elements of $H^1(U, \mathcal{O})$.
\end{proof}

\begin{corollary}
Let $X=\C^n/\Lambda$ be a Cousin group. Then  $\Lambda$ is strongly dispersive if and only if $H^1(U,\mathcal{O})$ is finitely generated for every convex domain $U$ in $X$. 
\end{corollary}


\section{Dolbeault cohomology of Oeljeklaus-Toma manifolds}\label{section:OT}

In this section we will apply Theorem~\ref{main} to determine the Dolbeault cohomology of Oeljeklaus-Toma manifolds.
We start by a brief presentation of their construction  following \cite{OT}.

Let $\Q \subseteq K$ be an algebraic number field with $n$ embeddings in $\C$, out of which $s$ are real, $\sigma_1,\ldots,\sigma_s\colon K\to \R$, and $2t$ are complex conjugated embeddings, $\sigma_{s+1},\ldots,\sigma_{s+t},\sigma_{s+t+1}=\overline\sigma_{s+1},\ldots,\sigma_{s+2t}=\overline\sigma_{s+t} \colon K\to\C$. Clearly, $n=s+2t$. In the sequel we shall only consider algebraic number fields for which $s, t \geq 1$.

Let $\mathcal O_K$ be the ring of algebraic integers of $K$, and $\mathcal O_K^{*,+}$ be the group of totally positive units, which is the subset of $\mathcal O_K$ consisting of those units with positive image through all the real embeddings. 

Consider the action $\mathcal{O}_K \circlearrowleft \mathbb H^s\times\C^t$ given by:
$$ T_a(w_1,\ldots,w_s,z_{s+1},\ldots,z_{s+t}) := (w_1+\sigma_1(a),\ldots,z_{s+t}+\sigma_{s+t}(a)) , $$
where $\mathbb{H}$ denotes the upper half-plane and the action $\mathcal{O}_K^{*,+} \circlearrowleft \mathbb H^s\times\C^t$ given by dilatations,
$$ R_u(w_1,\ldots,w_s,z_{s+1},\ldots,z_{s+t}) := (w_1\cdot \sigma_1(u),\ldots,z_{s+t}\cdot \sigma_{s+t}(u)) . $$
In \cite{OT} it is shown that there always exists a subgroup $U\subseteq\mathcal{O}_K^{*,+}$ such that the action $\mathcal{O}_K\rtimes U \circlearrowleft \mathbb{H}^s\times\C^t$ has no fixed point, is properly discontinuous and co-compact. We shall call such a subgroup {\it admissible}. The {\em Oeljeklaus-Toma manifold} (OT, shortly) associated to the algebraic number field $K$ and to the admissible subgroup of positive units $U$ is
$$ X(K,U) \;:=\; \left. \mathbb H^s\times\C^t \middle\slash \mathcal O_K\rtimes U. \right.$$
Note that the admissibility of $U$ is equivalent to the fact that the action of $U$ on $\mathbb{R}^s_{>0}$ by $u \cdot (r_1, \ldots, r_s)=(\sigma_1(u)r_1, \ldots, \sigma_s(u)r_s)$ is properly discontinuous and co-compact.  By  construction, $X(K, U)$ is a smooth fibre bundle over $\mathbb{R}_{>0}^s/U$, which is diffeomorphic to a real $s$-dimensional torus $\mathbb{T}^s$. Moreover, the fibre is again a real torus:
\begin{equation}\label{fibrat}
\mathbb{T}^{s+2t} \rightarrow X(K, U) \xrightarrow{\pi} \mathbb{T}^s,
\end{equation}
but the fibration is not principal  (and of course not holomorphic). The submersion map $\pi$ is given by:
\begin{equation}\label{submersion}
\pi ([w_1, \ldots, w_s, z_1, \ldots, z_t])=\widehat{(\Im w_1, \ldots, \Im w_s)}.
\end{equation}
  We call $X(K, U)$ {\it of simple type} if there exists no proper intermediate extension $\mathbb{Q} \subset K^' \subset K$ such that $U \subset O^*_{K^'}$.

By \cite[Lemma 2.4]{OT} $\C^{s+t}/\mathcal{O}_K$ is a Cousin group and 
 by Proposition~\ref{prop:algebraic_lattices} one has
\begin{proposition}\label{prop:condition_S_for_OT}
The discrete subgroup $\mathcal{O}_K$ is strongly dispersive.
\end{proposition}

Theorem \ref{main} will be applied to the convex open subset $(\H^s\times\C^t)/\mathcal{O}_K$ of the Cousin group   $\C^{s+t}/\mathcal{O}_K$.

{\it Warning:}
In the previous section, we denoted by $U$ a convex domain in a Cousin group, but for the rest of the exposition, $U$ shall only stand for an admissible group of positive units. Also $n$ equals now $s+2t$ and no longer denotes the dimension of the Cousin group we consider.
From now on, we use the notation $X$ instead of $X(K, U)$ for the Oeljeklaus-Toma manifold and not for the Cousin group $\C^{s+t}/\mathcal{O}_K$.

\begin{remark}\label{rem:generalizedOT}
In \cite{Moosa-Toma} the construction of Oeljeklaus-Toma manifolds was slightly generalized by replacing the discrete subgroup $\mathcal{O}_K$ by an additive subgroup $M$ of rank $s+2t$ which is stable under the action of $U$. The resulting manifolds $X(M,U)$ were shown to admit finite unramified covers of type $X(\mathcal{O}_K, U)$. All our results extend without difficulty to this larger class of compact complex manifolds. When $s=t=1$ the class of manifolds of type $X(M,U)$ coincides with the class of Inoue-Bombieri surfaces, \cite[Remark 8]{Moosa-Toma}.
\end{remark}

By the Dolbeault isomorphism, $H^{p,q}_{\overline{\partial}} (X)\simeq H^q(X, \Omega^p)$, where $\Omega^p$ is the sheaf of germs of holomorphic $p$-forms.

We shall compute $H^q(X, \Omega^p)$ by using three instruments: the Leray-Serre spectral sequence associated to the fibration \eqref{fibrat}, Theorem~\ref{main} and Fr\" olicher-type inequalities.

We denote by $^{p}E_r^{\cdot, \cdot}$ the Leray-Serre spectral sequence associated to \eqref{fibrat} and the sheaf $\Omega^p$. 
Then $^{p}E_2^{i, j}=H^{i}(\mathbb{T}^s, R^j\pi_*\Omega^p)$, where $R^j\pi_*\Omega^p$ is the sheafification of the presheaf $\mathcal{T}^p_j$ given by:
\begin{equation}
\mathcal{T}^p_j(W)=H^{j}(\pi^{-1}(W), \Omega^p_{|\pi^{-1}(W)}),
\end{equation}
for any open set $W$ of $\mathbb{T}^s$. We use the notation $\hat{\mathcal{T}^{p}_j}$ from now on, instead of $R^j \Omega^{p}$.

\begin{lemma}\label{Lem:sigma}
For any $0 \leq p, j \leq s+t$, $\hat{\mathcal{T}^{p}_j}$ is the local system on $\mathbb{T}^s$ associated to the representation $\rho: U \rightarrow GL\left(N(p, j), \C \right)$,
\begin{equation}
\rho(u)=diag(\sigma_I(u)\overline{\sigma}_{J}(u)),
\end{equation}
where $N(p, j)={s+t \choose p}{t \choose j}$, $I$ runs through all the subsets of length $p$ of $\{1, \ldots, s+t\}$, $J$ through all the subsets of length $j$ of $\{1, \ldots, t\}$ and for any $K=\{i_1, \ldots i_k\} \subseteq \{1, \ldots, s+t\}$, $\sigma_K(u):=\sigma_{i_1}(u) \cdot \ldots \cdot \sigma_{i_k}(u)$, with the convention that if $K \subseteq \{1, \ldots, t\}$,  $\overline{\sigma}_K(u):=\overline{\sigma_{s+i_1}(u)} \cdot \ldots \cdot \overline{\sigma_{s+i_k}(u)}$.
\end{lemma}

\begin{remark} In particular, when $j>t$, $\hat{\mathcal{T}^{p}_j}$ is the sheaf that vanishes on every open set of $\mathbb{T}^s$. Note that $\pi_1(\mathbb{T}^s)=U$.
\end{remark}
\begin{proof}
We show first that $\hat{\mathcal{T}^{p}_j}$ is locally constant, namely, that for any $x \in \mathbb{T}^s$, there exists an open set $W \ni x$ such that $\hat{\mathcal{T}^{p}_j}_{|W}$ is constant. Indeed, let $W \ni x$ be a trivialization open set for \eqref{fibrat},  which can be assumed, by shrinking it if needed, to be the image of an open set $W_1 \times \cdots \times W_s \subset \R_{>0}^{s}$, where each $W_j \subset R_{>0}$ is an open interval. Then for every $1 \leq j \leq s$ we set $\tilde{W}_i:= \mathbb{R} \times i W_j\subset\H$. 
This is an open convex set in $\mathbb{H}$. We further set  $\tilde{W}:=\tilde{W}_1 \times \ldots \times \tilde{W}_s\subset\mathbb{H}^s$.
By \eqref{submersion} $\pi^{-1}(W)= \tilde{W} \times \C^t / \mathcal{O}_K$. 
Since $\tilde{W} \times \C^t$ is an open convex $\mathcal{O}_K$-invariant subset of $\C^{s+t}$ and since $\mathcal{O}_K$ 
is strongly dispersive 
by Proposition~\ref{prop:condition_S_for_OT}, we are in a situation where Theorem~\ref{main} applies. 

By Theorem~\ref{main} applied to $\tilde{W} \times \C^t$ and $\mathcal{O}_K$, we get that for any open convex trivialization set $W$ 
\begin{equation}
\mathrm{dim}_{\C} H^j(\pi^{-1}(W), \Omega^p_{|\pi^{-1}(W)})={s+t \choose p}{t \choose j}=N(p, j).
\end{equation}

The basis of $H^j(\pi^{-1}(W), \Omega^p_{|\pi^{-1}(W)})$ is, therefore, given by $\{[d z_I \wedge d \overline{z}_J] \mid |I|=p,  |J| =j, I \subseteq \{1, \ldots, s+t\}, J \subseteq \{1, \ldots, t\}\}$.

Since the set of convex open sets $W$ is co-final, in the sense that 
$$(\mathcal{T}^p_j)_x = \mathop{\lim_{\longrightarrow}}_{V \ni x} \mathcal{T}^p_j(V)=\mathop{\lim_{\longrightarrow}}_{W \ni x, \\ W convex} \mathcal{T}^p_j(W),$$
we have $(\hat{\mathcal{T}^p_j})_x = (\mathcal{T}^p_j)_x=\C^{N(p, j)}$, meaning that $\hat{\mathcal{T}^p_j}$ is locally constant.
 
In order to determine the corresponding representation of $U$, we need to check how an element $u \in U$ acts on the basis $[d z_I \wedge d \overline{z}_J]$. From the definition of $OT$-manifolds, we have:
\begin{equation*}
u^*(dz_I \wedge d \overline{z}_J)= \sigma_I(u)\overline{\sigma}_J(u)dz_I \wedge d \overline{z}_J
\end{equation*} 
and consequently the representation associated to $\hat{\mathcal{T}^p_j}$ is precisely $\rho$. Since $\rho$ is diagonal, we deduce moreover that 
\begin{equation}\label{splitare}
\hat{\mathcal{T}^p_j} = \bigoplus_{I, J} L_{I, J}, 
\end{equation}
where $L_{I, J}$ is the flat complex line bundle over $\mathbb{T}^s$ associated to the representation $\rho_{I, J}: U \rightarrow \C^*$, $\rho_{I, J}(u)=\sigma_I(u)\overline{\sigma}_J(u)$. 
\end{proof}

\begin{theorem}\label{Hodgedecomposition}
Any OT-manifold $X$ 
satisfies the Hodge decomposition, in the sense that 
\begin{equation*}
\mathrm{dim}_{\C}H_{dR}^l(X)=\sum_{p+q=l} \mathrm{dim}_{\C}H^q(X, \Omega^p).
\end{equation*}
\end{theorem}

\begin{proof}
By the Fr\" olicher inequality, we have:
\begin{equation}\label{ineq1}
\mathrm{dim}_{\C}H_{dR}^l(X) \leq \sum_{p+q=l} \mathrm{dim}_{\C}H^q(X, \Omega^p).
\end{equation} 
Since $^{p}E_r^{\cdot, \cdot} \Rightarrow H^*(X, \Omega^p)$, by a Fr\" olicher type inequality, we get:
\begin{equation}\label{ineq2}
\mathrm{dim}_{\C} H^q(X, \Omega^p) \leq \sum_{i+j=q} \mathrm{dim}_{\C} H^i(\mathbb{T}^s, \hat{\mathcal{T}^p_j}).
\end{equation}

By \eqref{splitare}, $\mathrm{dim}_{\C} H^i(\mathbb{T}^s, \hat{\mathcal{T}^p_j})= \mathrm{dim}_{\C} H^i(\mathbb{T}^s, \bigoplus_{I, J}L_{I, J})$ and using now Lemma 2.4 in \cite{io}, the following holds:

\begin{align}\label{eq1}
\mathrm{dim}_{\C} H^i(\mathbb{T}^s, \bigoplus_{I, J}L_{I, J}) & = 
\mathrm{dim}_{\C} H^i(\mathbb{T}^s) \cdot \sharp\{I \subseteq \{1, \ldots, s+t\}, J \subseteq \{1, \ldots, t\} \mid |I|=p, |J|=j, \rho_{I, J} \equiv 1\}\\
& = {s \choose i} \cdot \sharp\{I \subseteq \{1, \ldots, s+t\}, J \subseteq \{1, \ldots, t\} \mid |I|=p, |J|=j, \sigma_I(u) \cdot \overline{\sigma}_{J}(u) \equiv 1\}
\end{align}
Putting together \eqref{ineq1}, \eqref{ineq2} and \eqref{eq1}, we have:
\begin{dmath}\label{dimensiunegrupuri}
\mathrm{dim}_{\C}H_{dR}^l(X) \leq  \sum_{p+q=l} \mathrm{dim}_{\C}H^q(X, \Omega^p) \leq 
\\
\sum_{p+q=l} \sum_{i+j=q}{s \choose i} \cdot \sharp\{I \subseteq \{1, \ldots, s+t\}, J \subseteq \{1, \ldots, t\} \mid |I|=p, |J|=j, \sigma_I(u) \cdot \overline{\sigma}_{J}(u) \equiv 1\}
\end{dmath}

Using the fact that for any $1 \leq r \leq t$, $\overline{\sigma_{s+r}(u)}=\sigma_{s+t+r}(u)$, the last term of the inequality can be rewritten as $\sum_{p+q=l} {s \choose q} \cdot \sharp\{I=\{i_1, \ldots, i_p\} \subseteq \{1, \ldots, s+2t \} \mid \sigma_{i_1}(u)\cdot \ldots \cdot \sigma_{i_p}(u)=1, \forall u \in U\}$. 
By Theorem 3.1 in \cite{io}, this is exactly $\mathrm{dim}_{\C} H^l_{dR}(X)$. Hence all the inequalities above are actually equalities and we obtain Hodge decomposition.
\end{proof}

\begin{corollary}\label{cor:H1}
For any OT manifold $X$ of type $(s, t)$, $\mathrm{dim}_{\C}H^1(X, \mathcal{O})=s$. 
\end{corollary}
\begin{proof}
By \cite{OT}, we know that $b_1=s$ and $H^0(X, \Omega^1)=0$. Applying now Theorem~\ref{Hodgedecomposition} for $l=1$, we immediately get 
$\mathrm{dim}_{\C}H^1(X, \mathcal{O})=s$.
\end{proof}

In \cite{io} it is shown that the de Rham cohomology of an OT manifold $X$ can be easily computed if $X$ satisfies the following condition:

{\em Condition (C): $\sigma_I \overline{\sigma}_J \equiv 1$ if and only if $I=J=\emptyset$ or $I=\{1, \ldots, s+t\}$ and $J=\{1, \ldots, t\}$, where $\sigma_I$ and $\overline{\sigma}_J$ are defined on $U$ as in Lemma \ref{Lem:sigma}.}

\begin{remark}
From a number theoretical point of view, Condition (C) means that if a product of embeddings $\sigma_L$ equals $1$ on $U$ then this must be a power of 
 $\sigma_1 \cdot \ldots \cdot \sigma_{s+2t}$. Moreover, it automatically implies that $X(K, U)$ is of simple type. Indeed, if there existed $\mathbb{Q} \subset K^{'} \subset K$ an intermediate extension such that $U \subset K^{'}$, then there would exist a product  $\sigma_{i_1} \cdot \ldots \cdot \sigma_{i_k}$ giving $1$ on $U$, where $k=[K^{'}:\mathbb{Q}]<s+2t$.

A specific situation when Condition (C) holds is when $|\sigma_{s+1}(u)|^2= \ldots = |\sigma_{s+t}(u)|^2$, for any $u \in U$, as shown in the proof of \cite[Proposition 6.4]{io}. Geometrically, this particular condition is equivalent to the existence of a {\it locally conformally K\" ahler metric} on $X(K, U)$, as proven in \cite[Proposition 2.9]{OT} and in the appendix of L. Battisti of \cite{Dub}, Theorem 8.

On the topological side, Condition (C) can also be interpreted as $X(K, U)$ having the lowest possible Betti numbers among the OT-manifolds of type $(s, t)$, see \cite[Theorem 3.1]{io}. 
\end{remark}

By a straightforward computation that results from \eqref{dimensiunegrupuri} being an equality, we also have the following:

\begin{corollary}\label{calcul}
If $X$ satisfies Condition (C), then $$\mathrm{dim}_{\C} H^q(X, \mathcal{O})={s \choose q} \  \text{if} \ q \leq s, \ \mathrm{dim}_{\C} H^q(X, \Omega^{s+t})={s \choose q-t} \ \text{if} \  q \geq t$$ and the rest of the Dolbeault cohomology groups are trivial.
\end{corollary}

\begin{remark}
In \cite{kasuya} it is shown that any OT manifold $X$ admits a solvmanifold presentation $\Gamma \setminus G$, in such a way that the natural complex structure on $G$ is $G$-left invariant. It is well known that the Lie algebra cohomology $H^*(\mathfrak{g})$ injects into $H^*_{dR}(X)$. If $X$ satisfies Condition (C), one can check that the generators given in \cite{io} are $G$-invariant, hence the inclusion morphism $H^*(\mathfrak{g}) \to H^*_{dR}(X)$ is an isomorphism. This and the Hodge decomoposition for $X$ gives an isomorphism at the level of Dolbeault cohomologies $H^{*, *}(\mathfrak{g}) \cong H^{*, *}(X)$.    
\end{remark}

\begin{corollary}
If $X$ is of simple type, then $H^0(X, \Omega^2)=0=H^1(X, \Omega^1)$ and $\mathrm{dim}_{\C}H^2(X, \mathcal{O})={s \choose 2}$.
\end{corollary}

\begin{proof}
By the proof of \cite[Proposition 2.3]{OT}, we deduce that if $X$ is of simple type, then for any different indices $i_1, i_2 \in \{1, \ldots, s+2t\}$, $\sigma_{i_1}\sigma_{i_2}: U \rightarrow \C^*$ is not trivial and moreover, $b_2={s \choose 2}$. Therefore, using again \eqref{dimensiunegrupuri} for $l=2$, we obtain the stated dimensions.
\end{proof}

Finally by Corollary \ref{cor:H1} and Remark \ref{rem:generalizedOT} we get

\begin{corollary}\label{cor:inoue}
For an Inoue-Bombieri surface  $X$ one has $\mathrm{dim}_{\C}H^1(X, \mathcal{O})=1$. 
\end{corollary}


\end{document}